\documentclass{amsart}

 \newtheorem{thm}{Theorem}[section]
 
 \newtheorem{lem}[thm]{Lemma}
 
 \theoremstyle{definition}
 
 \theoremstyle{remark}

 \usepackage{enumitem}
 \usepackage{color}

\begin{document}

\title[Generalized derivations with central values on Lie ideals]
 { Generalized derivations with central values on Lie ideals}

\author{shervin sahebi, venus rahmani$^{*}$ }

\address{ $^{*}$ Department Of Mathematics, Islamic Azad University,
          Central Tehran Branch, 13185/768, Tehran, Iran.}

\email{$sahebi@iauctb.ac.ir, ven.rahmani.sci@iauctb.ac.ir$}

\thanks{}

\thanks{}






\begin{abstract}
 let $R$ be a prime ring of $char R \neq 2$, $H$ a
 generalized derivation and  $L$ a noncentral
 lie ideal of $R$. We show that if $l^s H(l) l^t\in Z(R)$
 for all $l\in L$, where $s,t\geq 0$  are fixed integers,
 then $H(x)=bx$ for some $b\in C$, the extended centroid
 of $R$, or $R$ satisfies $S_4$. Moreover, let $R$ be a
 $2$-torsion free semiprime ring, let $A=O(R)$ be an orthogonal
 completion of $R$ and $B=B(C)$  the Boolean ring of $C$.
 Suppose $([x_1,x_2]^sH([x_1,x_2])[x_1,x_2]^t\in Z(R)$ for
 all $x_1, x_2 \in R$, where $s,t\geq 0$ are fixed integers.
 Then there exists idempotent $e\in B$ such that $H(x)=bx$
 on $eA$ and the ring $(1-e)A$ satisfies $S_4$.\\ \\

\end{abstract}

\maketitle

\section{Introduction}
 Let $R$ be an associative ring with center $Z(R)$.
 Recall that an additive map $d: R\rightarrow R$
 is called  derivation if $d(xy)=d(x)y+xd(y)$,
 for all $x,y \in R$.
 Many results in literature indicate that global
 structure of a prime(semiprime) ring $R$ is
 often lightly connected to the behaviour of
 additive mappings defined on $R$.
 A well-known result of Herstein ~\cite{a06} stated that
 if $d$ is a nonzero derivation of a prime ring
 $R$ such that $d(x)^n\in Z(R)$ for all $x\in R$,
 then $R$ satisfies $S_4$, the standard identity
 in four variables.
 Herstein's result was extended to the case of
 Lie ideals of prime rings by Bergen and Carini~\cite{a001}.
 Some articles was studied derivation with
 central values on Lie ideals~\cite{a1, a999}.
 Recently, Dhara ~\cite{a22} studied the more
 generalized situation when $l^sd(l)l^t\in Z(R)$,
 for all $l\in L$, the noncentral Lie ideal of $R$,
 where $s, t \geq 0$
 are some fixed integers.

 \noindent
 Here we will consider the same situation in case
 the derivation $d$ is replaced by generalized derivation $H$.
 More specifically an additive map
 $H: R\rightarrow R$ is
 called  generalized derivation if there is
 a derivation $d$ of $R$  such that $H(xy)=H(x)y+xd(y)$,
 for all $x, y \in R$.

\noindent
 Throughout the paper we use the standard notation
 from \cite{a0}. In particular, we denote by  $Q$
 the two sided Martindale quotient of prime(semiprime) ring $R$
 and $C$ the center of $Q$. We call $C$ the extended centroid
 of $R$.\\

 \noindent
 The main results of this paper are as follows:
 \begin{thm}\label{f5}
 Let $R$ be a prime ring of $char R\neq 2$, $H$
 generalized derivation and $L$ a noncentral
 Lie ideal of $R$. Suppose $l^s H(l)l^t\in Z(R)$
 for all $l\in L$, where $s, t\geq 0$,
 are fixed integers. Then $H(x)=bx$ for
 some  $b\in C$, the extended centroid of $R$,
 or $R$ satisfies $S_4$.
 \end{thm}

 \noindent
 When $R$ is a semiprime ring, we prove:
 \begin{thm}\label{THM2}

 let $R$ be a $2$-torsion free semiprime ring with
 generalized derivation $H$.  Consider
 $[x_1,x_2]^s H([x_1,x_2]) [x_1,x_2]^t \in Z(R)$
 for all $x_1, x_2\in R$,
 where $s, t\geq 0$ are fixed integers.
 Further, let $A=O(R)$ be the orthogonal completion
 of $R$ and $B=B(C)$ where $C$ the extended centroid
 of $R$. Then there exists idempotent $e\in B$ such
 that $H(x)=bx$ on $eA$ and the
 ring $(1-e)A$ satisfies $S_4$.
\end{thm}

\section{proof of the main results}
  \noindent
  The following results are useful tools
  needed in the proof of the main results.

  \begin{lem}\label{f0}
  Every generalized derivation $H$ on a dense right ideal of
  prime(semiprime) ring  $R$ can be uniquely extended to a
  generalized derivations of $Q$. Also  can be write in the form
  $H(x)=bx+d(x)$ for some $b\in Q$, all $x\in Q$ and  a derivation
  $d$ of $Q$~\cite{a10}.
  \end{lem}
  \begin{lem}\label{f1}
  \emph{(see \cite[Lemma 2]{a9} and \cite[Lemma 1]{a01}).}
  Let $R$ be a prime ring  of $char R\neq 2$,
  $L$ be a noncentral Lie ideal
  of $R$  and $I$ be the ideal of $R$ generated by
  $[L,L]$. Then $I\subseteq L+L^2$ and $[I,I]\subseteq L$.
  \end{lem}
 \begin{thm}\label{f2}
 \emph{( Kharchenko ~\cite{a8})}.
 Let $R$ be a prime ring, $d$
 a nonzero derivation of $R$ and $I$ a nonzero ideal of $R$.
 If $I$ satisfies the differential identity

 \noindent
 $$f(r_{1},r_{2},\ldots,r_{n},d(r_{1}),d(r_{2}),\ldots,d(r_{n})) = 0,$$
 for any $r_{1},r_{2},\ldots,r_{n}\in I$, then one of the following holds:

 \noindent
\begin{enumerate}[label=(\roman{*})]
   \item  $I$ satisfies the generalized polynomial identity
   $$f(r_{1},r_{2},\ldots,r_{n},x_{1},x_{2},\ldots,x_{n}) = 0.$$

   \item $d$ is $Q$-inner, that is, for some $q\in Q$,
   $d(x) = [q,x]$ and $I$ satisfies the generalized polynomial identity

  \noindent
  $$f(r_{1},r_{2},\ldots,r_{n},[q,r_{1}],[q,r_{2}],\ldots,[q,r_{n}]) = 0.$$
  \end{enumerate}
  \end{thm}

 \noindent
  We establish the following technical results required
  in the proof of Theorem \ref{f5}.
\begin{lem}\label{f3}
  Let $R=M_k(F)$ be a ring  of all $k\times k$ matrices over a
  field $F$ where $k\geq 3$. Suppose $b[x_1,x_2]+[x_1,x_2]c\in Z(R)$
  for some $b, c \in R$ and all $x_1, x_2 \in R$.
  Then $b, c \in F\cdot I_{k}.$
 \end{lem}
 \begin{proof}
 Let $b=(b_{ij})_{k\times k}$, $c=(c_{ij})_{k\times k}$.
 Putting $x_1=e_{11}$ and  $x_2=e_{12}$, we obtain
 $b[x_1,x_2]+[x_1,x_2]c=be_{12}+e_{12}c.$
 Since rank of $b[x_1,x_2]+[x_1,x_2]c\leq 2$, it can not
 be invertible. This implies $be_{12}+e_{12}c=0$.
 Left and right  multiplying  by $e_{12}$, we get
 $$\begin{array}{c}
  0=e_{12}(be_{12}+e_{12}c)=b_{21}e_{12},\\
  0=(be_{12}+e_{12}c)e_{12}=c_{21}e_{12}.
  \end{array}$$\\
  This implies that $c_{21}=b_{21}=0$.
  Thus for any $i\neq j $,
  $b_{ij}=c_{ij}=0$. That is, $b$ and $c$ are diagonal.
  Let $b=\sum_{i=1}^k b_{ii}e_{ii}$.
  For any $F$-automorphism
  $\theta$ of $R$ $b^{\theta}$ enjoys the same property as $b$
  does, namely, $b^{\theta}[x_1,x_2]+[x_1,x_2]c^{\theta}$ is
  zero or invertible, for every $x_1, x_2 \in R$. Hence $b^{\theta}$
  must be diagonal. Then for each $j\neq 1$,

\noindent
 $$(1+e_{1j})b(1-e_{1j})=\sum_{i=1}^k b_{ii}e_{ii}+(b_{jj}-b_{11})e_{1j},$$
 is diagonal. Therefore, $b_{jj}=b_{11}$ and so $b\in F\cdot I_k $.
 Similarly, we conclude $c\in F\cdot I_k$.
\end{proof}
 \begin{lem}\label{f4}
 Let $R=M_k(F)$ be a ring of all $k\times k$ matrices over a field $F$
 of  $char F\neq2$, where $k\geq 3$. Suppose
 $[x_1,x_2]^s(b[x_1,x_2]+[x_1,x_2]c)[x_1,x_2]^t\in Z(R)$,
 for some $b, c \in R$ and all $x_1, x_2\in R$ where
 $s, t\geq 0$ are fixed integers
 such that $s+t\neq 0$. Then $b, c\in F\cdot I_k$.
\end{lem}
\begin{proof}
 Let $b=(b_{ij})_{k\times k}$, $c=(c_{ij})_{k\times k}$ and set

\noindent
 $$f(x_1,x_2)=[x_1,x_2]^s(b[x_1,x_2]+[x_1,x_2]c)[x_1,x_2]^t.$$

\noindent
 Putting $x_1=e_{11}$, $x_2=e_{12}-e_{21}$, we obtain
 $[x_1,x_2]=e_{12}+e_{21}$  and $[x_1,x_2]^n=e_{11}+e_{22}$
 for $n\geq 2$.
 So we have four cases:\\

\noindent
\emph{ Case 1}. $s=t=1$. We get
 $$f(x_1,x_2)=(b_{21}+c_{12})e_{11}+ (b_{12}+c_{21})e_{22}+(b_{22}+c_{11})e_{12}+(b_{11}+c_{22})e_{21}.$$

 \noindent
 \emph{Case 2}. $s=0$ and $t=1$. We get
 $$f(x_1,x_2)=(b_{11}+c_{22})e_{11}+(b_{22}+c_{11})e_{22}+(b_{12}+c_{21})e_{12}+ (b_{21}+c_{12})e_{21}+\sum_{i=3}^{k}b_{i1}e_{i1}+\sum_{i=3}^{k}b_{i2}e_{i2}.$$

 \noindent
 \emph{Case 3}. $s=1$ and $t=0$. We get
 $$f(x_1,x_2)=(b_{22}+c_{11})e_{11}+(b_{11}+c_{22})e_{22}+(b_{21}+c_{12})e_{12}+(b_{12}+c_{21})e_{21}+
\sum_{i=3}^{k}c_{1i}e_{1i}+\sum_{i=3}^{k}c_{2i}e_{2i}.$$

 \noindent
 \emph{Case 4}. $s, t\geq 2$. We obtain
 $$f(x_1,x_2)=(b_{12}+c_{21})e_{11}+(b_{21}+c_{12})e_{22}+(b_{11}+c_{22})e_{12}+(b_{22}+c_{11})e_{21}.$$

 \noindent
 In each cases, since rank of  $f(x_1,x_2)$ $\leq 2$,  $f(x_1,x_2)=0$.
 Thus
 $$\begin{array}{ccc}
  b_{12}=-c_{21} & and & b_{21}=-c_{12},
 \end{array}$$

\noindent
and so for any $i\neq j$ we have
\begin{equation}\label{d1}
b_{ij}=-c_{ji}.
\end{equation}

\noindent
 Now putting $x_1=e_{11}$, $x_2=e_{12}+e_{21}$,
 we have  $[x_1,x_2]^n=(-1)^{n/2}(e_{11}+e_{22})$
 if $n$ is even and
 $(-1)^{(n-1)/2}(e_{12}-e_{21})$ if $n$ is odd.
 Four cases may be occured:\\

 \noindent
\emph{ Case 1}. s and t are even. We get
 $$f(x_1,x_2)=\pm((-b_{12}+c_{21})e_{11}+(b_{21}-c_{12})e_{22}+(b_{11}+c_{22})e_{12}+(-b_{22}-c_{11})e_{21}).$$

 \noindent
\emph{ Case 2}. s and t are odd. We get
 $$f(x_1,x_2)=\pm((-b_{21}+c_{12})e_{11}+(b_{12}-c_{21})e_{22}+(-b_{22}-c_{11})e_{12}+(b_{11}+c_{22})e_{21}).$$

 \noindent
 \emph{Case 3}. s is even and t is odd. We get
 $$f(x_1,x_2)=\pm((-b_{11}-c_{22})e_{11}+(-b_{22}-c_{11})e_{22}+(-b_{12}+c_{21})e_{12}+(-b_{21}+c_{12})e_{21}).$$

 \noindent
 \emph{Case 4}. s is odd and t is even. We get
 $$f(x_1,x_2)=\pm((-b_{22}-c_{11})e_{11}+(-b_{11}-c_{22})e_{22}+(b_{21}-c_{12})e_{12}+(b_{12}-c_{21})e_{21}).$$

 \noindent
 In each cases, since rank of $f(x_1,x_2)$  $\leq 2$,
 $f(x_1,x_2)=0$. Thus
 $$\begin{array}{ccc}
  b_{12}=c_{21} & and & b_{21}=c_{12},
 \end{array}$$

\noindent
 and so for any $i\neq j$ we have
\begin{equation}\label{d2}
 b_{ij}=c_{ji}.
\end{equation}

\noindent
(\ref{d1}) and (\ref{d2}) imply that $b$ and $c$ are diagonal.
 So we apply the same argument used in the proof of Lemma~\ref{f3}.
 Hence $b, c\in F\cdot I_k$.
 \end{proof}

\smallskip
\noindent
 Now we can prove Theorem~\ref{f5}.

\vspace{3mm}
 \emph{Proof of Theorem~\ref{f5}.} Since $char R\neq 2$ and $L$ is noncentral Lie ideal,
 by Lemma ~\ref{f1} there exists an ideal $I$ of $R$ such that
 $0\neq [I,I]\subseteq L$ and $[L,L]\neq 0.$
 Hence, without loss of generality, we may assume $L=[I,I]$.
 Thus $I$ satisfies the generalized differential identity
 $$[x_1,x_2]^s H([x_1,x_2])[x_1,x_2]^t\in Z(R).$$

 \noindent
 Let $Q$ be the two sided  Martindale quotient ring and $C$ the extended
 centroid of $R$. By~\cite{a10} $I$ and $Q$  satisfy the same
 differential identities, thus we may assume
 $$[x_1,x_2]^sH ([x_1,x_2]) [x_1,x_2]^t\in Z(R),$$

 \noindent
 for all $x_1, x_2\in Q$. By Lemma \ref{f0} we may assume
 $H(x)=bx+d(x)$ for some $b\in Q$, all $x \in Q$
 and  $d$ a derivation of $Q$. Hence $Q$ satisfies
 $$[x_1,x_2]^s(b[x_1,x_2]+d([x_1,x_2])) [x_1,x_2]^t\in Z(R).$$
 This is a polynomial identity. Hence there exists a field
 $F$ such that $Q\subseteq M_k(F)$, the ring of $k\times k$ matrices over field $F$,
 where $k>1$. Moreover $Q$ and
 $M_k(F)$ satisfy the same polynomial identity
 \cite{a99}. Hence we have
 \begin{equation}\label{d4}
 [x_1,x_2]^s(b[x_1,x_2]+d([x_1,x_2])) [x_1,x_2]^t\in Z(M_k(F)).
 \end{equation}
  Now consider  two cases.\\

  \emph{case 1}. $d$ is a $Q$-inner derivation.
  In this case, there exists  an element $p\in Q$ such that
  $d(x)=[p,x]$ for all $x\in M_k(F)$, then (\ref{d4})
  becomes
  $$[x_1,x_2]^s(b[x_1,x_2]+[p,[x_1,x_2]])[x_1,x_2]^t\in Z(M_k(F)).$$
  So
  $$[x_1,x_2]^s((b+p)[x_1,x_2]-[x_1,x_2]p)[x_1,x_2]^t\in Z(M_k(F)),$$
  for all $x_1, x_2\in M_k(F)$.
  In this case if $k\geq 3$ and $s=t=0$, then by Lemma \ref{f3}
  we have  $-p, b+p \in F\cdot I_k$.
  Also for $k\geq3$ and $s+t\neq 0$,  Lemma \ref{f4} implies
  $-p, b+p \in F\cdot I_k$. Then $b\in F\cdot I_k$, and so
  $d(x)=0$.
  Hence $H(x)=bx$ for all $x\in M_k(F)$.
  So by ~\cite{a99} for all $x\in R$ we have $H(x)=bx$.
  If $k=2$, then $R$ satisfies $S_4$.\\

 \emph{case 2}. $d$ is not a $Q$-inner derivation.
  In this case we have
  $$[[x_1,x_2]^s(b[x_1,x_2]+d([x_1,x_2])) [x_1,x_2]^t,x_3]=0,$$
  for all $x_1, x_2, x_3\in M_k(F)$.

 \noindent
  Then by Theorem \ref{f2} we have
  $$[[x_1,x_2]^s (b[x_1,x_2]+[x_4,x_2]+[x_1,x_5]) [x_1,x_2]^t,x_3]=0, $$
  for all $x_1, x_2, x_3, x_4, x_5 \in M_k(F)$.
  In particular, $M_k(F)$ satisfies its blended component
  $$[[x_1,x_2]^s ([x_4,x_2]+[x_1,x_5]) [x_1,x_2]^t,x_3]=0.$$

 \noindent
 If $k\geq 3$, then by choosing
 $$\begin{array}{ccccc}
    x_1=e_{ij}, & x_2=e_{ji}, & x_3=e_{ik}, & x_4=e_{ij}, & x_5=0,
 \end{array}$$

 \noindent
 for all $i\neq j\neq k$, we get
 $$0=[[x_1,x_2]^s ([x_4,x_2]+[x_1,x_5]) [x_1,x_2]^t,x_3]=e_{ik},$$
 which is a contradiction.
 Thus  $k=2$, that is, $R$ satisfies $S_4$.\hfill $\Box$\\

 Now let $R$ be a semiprime orthogonally complete ring with
 extended centeroid $C$. The notations $B=B(C)$
 and $spec(B)$ denotes Boolian ring of $C$ and the
 set of all maximal ideal of $B$, respectively.
 It is well known that if $M\in spec(B)$ then $R_{M}=R/RM$
 is prime ~\cite[Theorem 3.2.7]{a0}. We use the notations
 $\Omega$-$\Delta$-ring,
 Horn formulas and Hereditary formulas. We
 refer the reader to ~\cite[ pages 37, 38, 43, 120]{a0}
 for the definitions and the related properties
 of these objects.\\

\noindent
 We establish the following technical result required
 in the proof of Theorem \ref{THM2}.

\begin{lem}\label{f6}
 \cite[Theorem 3.2.18]{a0}. Let $R$ be an orthogonally complete
 $\Omega$-$\Delta$-ring with extended centroid $C$,
 $\Psi_{i} ( x_{1}, x_{2},\ldots, x_{n})$ Horn formulas of signature
 $\Omega$-$\Delta$, $i=1,2,\ldots$ and
 $\Phi(y_{1}, y_{2},\ldots,y_{m})$
 a Hereditary first order formula such that $\neg\Phi$ is
 a Horn formula.
 Further, let
 $\vec{a}= (a_{1}, a_{2},\ldots,a_{n})\in R^{(n)},$
 $\vec{c} = (c_{1}, c_{2},\ldots, c_{m})\in R^{(m)}.$
 Suppose  $R\models \Phi (\vec{c})$ and  for
 every $M\in spec (B)$ there exists a natural number
 $i=i(M)>0$ such that
 $$R_{M} \models \Phi (\phi_{M} (\vec{c})) \Longrightarrow \Psi_{i}(\phi_{M}(\vec{a})),$$\\
  where $\phi_{M}: R\rightarrow R_{M} = R/RM$ is the
  canonical projection. Then there exists a natural number
  $k > 0$ and pairwise orthogonal idempotents
  $e_{1}, e_{2},\ldots,e_{k}\in B$ such that
  $e_{1} + e_{2} + \ldots + e_{k} = 1$ and $e_{i}R\models\Psi_{i}(e_{i}\vec{a})$
  for all $e_{i}\neq 0$.
 \end{lem}

 \noindent
 we denote $O(R)$ the orthogonal completion of $R$
 which is defined as the intersection of all
 orthogonally complete subset of $Q$ containing $R$.

\smallskip
\noindent
 Now we can prove Theorem~\ref{THM2}.

\vspace{3mm}
\emph{Proof of Theorem~\ref{THM2}.}
 By  assumption we have $R$  satisfies
 $$[[x_1,x_2]^sH([x_1,x_2]) [x_1,x_2]^t, x_3]=0. $$
 By Lemma ~\ref{f0} the generalized derivation $H$
 can be extended uniquely to the generalized derivation on $Q$,
 moreover, we  may assume
 $H([x_1,x_2])=b[x_1,x_2]+d([x_1,x_2])$,
 for some $b\in Q$, all $x_1, x_2 \in Q$ and
  $d$ a derivation of $Q$. Hence $Q$ satisfies
 $$[[x_1,x_2]^s(b([x_1,x_2]+d([x_1,x_2])) [x_1,x_2]^t, x_3]=0.$$

 \noindent
 According to \cite[Theorem 3.1.16]{a0} $d(A)\subseteq A$
 and $d(e)=0$ for all $e\in B.$ Therefore, $A$ is an orthogonally
 complete $\Omega$-$\Delta$-ring, where $\Omega= \{o, +, -, \cdot, d \}$.
 Consider formulas\\

 $$\begin{array}{l}
 \Phi =  (\forall x_1 )(\forall x_2 )  \| [[x_1,x_2]^s(b[x_1,x_2]+d([x_1,x_2]) [x_1,x_2]^t, x_3]=0 \|,\\ \\
 \Psi_{1} = (\forall x )  \| H(x)=bx\|,\\ \\
 \Psi_{2} = (\forall x_1)(\forall x_2)(\forall x_3)( \forall x_4 ) \| S_4(x_1, x_2, x_3, x_4)=0 \|.
 \end{array}$$\\ \\
 We can easily check that $\Phi $ is a hereditary first order
 formula and $\neg \Phi$, $\Psi_{1}$, $\Psi_{2}$ are Horn formulas.
 So using Theorem ~\ref{f5}, all conditions of Lemma ~\ref{f6}
 are fulfilled. Hence there exist two orthogonal
 idempotents $e_{1}$ and $e_{2}$ such that $e_{1} + e_{2} = 1.$
 If $e_{i} \neq 0$, then $e_{i}A \models \Psi_{i},$ $i = 1, 2.$
 This complete the proof.\hfill $\Box$


 \end{document}